\documentclass{amsart}
\usepackage{graphicx}
\usepackage{amssymb}
\newtheorem{thm}{Theorem}[section]
\newtheorem{cor}[thm]{Corollary}
\newtheorem{lem}[thm]{Lemma}

\theoremstyle{definition}
\newtheorem{defn}[thm]{Definition}
\newtheorem{remark}[thm]{Remark}
\theoremstyle{question}
\newtheorem{que}[thm]{Question}
\theoremstyle{Conjecture}

\theoremstyle{Problem}

\numberwithin{equation}{section}
\newcommand{\Real}{\mathbb N}

\begin{document}

\title[ an affirmative answer]{ an affirmative answer to the  Jaikin-Zapirain's question}%
\author{M. Zarrin}%

\address{Department of Mathematics, University of Kurdistan, P.O. Box: 416, Sanandaj, Iran}%
 \email{ M.zarrin@uok.ac.ir}
% ----------------------------------------------------------------
\begin{abstract}
If $X$ is a non-empty subset of a finite group  $G$, we denote by $o(x)$ the order of $x$ in $G$.  Then we put
$$o(X)=\frac{\sum_{x \in X} o(x)}{\mid X \mid}.$$
The number $o(X)$ is called the average order of $X$. Zapirain in 2011 \cite{zapir}, posed the following question:\\
Let $G$ be a finite {\rm (}$p$-{\rm)} group and $N$ a normal {\rm(}abelian{\rm }) subgroup of $G$. Is it true that $o(G)\geq o(N)^{1/2}$ {\rm ?}
Here, we will improve his question and confirm it.\\

%{\bf Mathematics Subject Classification (2010)}. ....
\end{abstract}
\maketitle
% ----------------------------------------------------------------

\section{\textbf{ Introduction}}

If $X$ is a non-empty subset of a finite group  $G$, we denote by $o(x)$ the order of $x$ in $G$.  Then we put
$$o(X)=\frac{\sum_{x \in X} o(x)}{\mid X \mid}=\frac{\psi(X)}{\mid X \mid}$$
The number $o(X)$ is called the average order of $X$.  
Since there is an interesting relation between the number of conjugacy classes of a finite group, say $k(G)$ with  its average order, Zapirain in 2011 \cite{zapir}, considered the average order for some powerful $p$-group of exponent $p^t$ and posed the following question:

\begin{que}\label{q}
Let $G$ be a finite {\rm (}$p$-{\rm)} group and $N$ a normal {\rm(}abelian{\rm }) subgroup of $G$. Is it true that $o(G)\geq o(N)^{1/2}$ {\rm ?}
\end{que}

 In this paper, we will improve his question and prove it. In fact, we will show that  for every finite group $G$ and  its special subset (including its subgroups), that it will be called $CC$-subsets, his question is true. 
  
 \begin{defn}\label{q}
A non-empty subset $C$ of a group  $G$ is called $CC$-subset (co-prime power closed), if $a\in C$ and $(n, o(a))=1$ then $a^n\in C$.   
\end{defn}

Our main result is the following.
\begin{thm}
Let $G$ be a finite  group and $A$ a $CC$-subset of $G$. Then $o(G)\geq o(A)^{1/2}$.
\end{thm}

\section{\textbf{ The Proof}}

For proof the main theorem, we need to define   equivalence relation $\frak{R_1}$  on $G$ as below:

$$ \forall ~~ g,h \in G  \quad g~\frak{R_1}~ h \hspace{3mm} if ~and~only ~if \hspace{3mm} \exists ~n \in \Real ~~such ~~that~~(n,o(g))=1 ~~and ~~ h=g^n .$$

\begin{lem}\label{l4}
Suppose that  $\bar x_i$ are some of the equivalence classes  with respect to the relation $\frak{R_1}$ on $G$ such that $o(x_t)\leq o(x_i)$ for 
$1\leq i \leq t$, then  $$o(\bar x_1 \cup \bar x_2 \cup \dots \cup \bar x_{t-1})\geq o(\bar x_1 \cup \bar x_2 \cup \bar x_3 \cup \dots \cup \bar x_t).$$
\end{lem}
\begin{proof}
Put $\alpha=\sum_{i=1}^{t-1} \phi(o(x_i))o(x_i)$ and $\beta=\sum_{i=1}^{t-1}  \phi(o(x_i))$, where $\phi(n)$ is the Euler's totient function.  It follows that 
$$o(\bar x_1 \cup \bar x_2 \cup \dots \cup \bar x_{t-1})=\frac{\sum_{i=1}^{t-1} \psi(\bar x_i)}{\sum_{i=1}^{t-1} |\bar x_i|}=\frac{\sum_{i=1}^{t-1} \phi(o(x_i))o(x_i)}{\sum_{i=1}^{t-1} \phi(o(x_i))}=\frac{\alpha}{\beta}$$  and $o(\bar x_1 \cup \bar x_2 \cup \dots \cup \bar x_{t})=\frac{\alpha+\phi(o(x_t))o(x_t)}{\beta+\phi(o(x_t))}$. Now it is not hard to see that  $\frac{\alpha}{\beta}\geq \frac{\alpha+\phi(o(x_t))o(x_t)}{\beta+\phi(o(x_t))}$ and the result follows. 
 \end{proof}

The following result would be independently interesting.  

\begin{lem}\label{l3}
If $G$ is finite group and $x\in G$, then  $o(G)\geq o(x)^{1/2}$. In particular,   $o(G)\geq meo(G)^{1/2}$, where $meo(G)$ is the maximum order of an element of $G$. 
\end{lem}
\begin{proof}
We put $\psi(G)=\sum_{g \in G} o(g)$ and $$E(G)= \{(a, b) \in G\times G \mid \text{there~ exists~ } i~\text{~in~} \mathbb{N} \text{~such ~that~} b=a^i  \text{~and~} 1\leq i \leq |a| \}.$$
Then it is easy to see that $$\psi(G)=\mid E(G) \mid= |\{(a, b)\in G\times G \mid  ~~b~ \in ~  <a> \} \mid.$$
(without considering $\psi(G)$ as the cardinality of the set, it seems to be very hard to prove this lemma).
To prove, it is enough to show that $\psi(G)^2 \geq |G|^2 o(a)$, where $meo(G)=o(a)$.     
That is, we show that $$|E(G)\times E(G)| \geq |G\times G\times <a>|.$$ In fact, we should show that the size of the set $$A=\{((x, x^i), (y, y^j)) \mid   x, y \text{~~belong ~~to~~} G   \text{~and~}~~~ 1\leq i \leq |x|,~~~1\leq j \leq |y|\}$$ is greater than the size of the set 
$$B=\{(z, w, a^t, 1) | ~~z, w \text{~~belong ~to}~  G ~~\text{and}~~ 1\leq t \leq |a| \}.$$ 

But it is not hard to find a injective function like $f$ from the set $B$ to the set  $A$ and the result follows.  
\end{proof}
 
Now we are ready to prove the main result.\\

$\mathbf{Proof~~ of ~~Theorem ~~1.3}.$

Let $G$ be a finite  group and $A$ a $CC$-subset of $G$.  By considering the relation $\frak{R_1}$ on $G$, we
can see that there exists $t\geq 1$ such that $A=\bar x_1 \cup \bar x_2 \cup \dots \cup \bar x_{t}$. We prove it, by induction on $t$.  According to Lemma \ref{l3}, the theorem is true for $t=1$.   Now without loss of generality we can assume that 
$o(x_t)\leq o(x_i)$ for 
$1\leq i \leq t$. Thus, by Lemma \ref{l4}., we have $o(A)\leq o(\bar x_1 \cup \bar x_2 \cup \dots \cup \bar x_{t-1})$ and so the result is followed by induction  hypothesis. 

\begin{remark}
By considering the conjugacy relation  on $G$, we obtain that $$\psi(G)=\sum_{i=1}^{k(G)} \psi((\bar x_i))= \sum_{i=1}^{k(G)}  |G|o(x_i)/|C_G(x_i)|$$ and so $$o(G)=\sum_{i=1}^{k(G)}o(x_i)/|C_G(x_i)|.$$ From this we can follow that:  
For every  finite  group $G$ we have $o(G)\leq k(G)$ {\rm (see also Corollary 2.10 of \cite{zapir})}.
\end{remark}

\begin{cor}\label{c3}
For every  finite  group $G$ we have $k(G)\geq meo(G)^{1/2}$.
\end{cor}

Finally,  we show that for some special subgroups of $G$, the Zapirain's Question will be improved.  
\begin{lem}\label{l5}
Let $G$ be a finite  group and $N$ a  subgroup of $G$. Then $$o(N\cap Z(G)) \leq o(G).$$
\end{lem}
\begin{proof}
 It is easy to see that  $\psi(G)=\sum_{i=1}^{[G:N\cap Z(G)]} \psi(a_i(N\cap Z(G)))$. Now as 
 $$\psi(a_i(N\cap Z(G))\geq \psi(N\cap Z(G)),$$ we obtain that  $$\psi(G)\geq \sum_{i=1}^{[G:N\cap Z(G)]} \psi(N\cap Z(G))= [G:N\cap Z(G)] \psi(N\cap Z(G).$$
 Thus $o(G)\geq o(N\cap Z(G))$.
\end{proof}
\begin{cor}\label{c3}
For every  finite  group $G$ we have $o(G)\geq o(Z(G))$ {\rm (see also Lemma 2.7 of \cite{zapir})}.
\end{cor}

\end{document}